\documentclass[a4paper,11pt]{article}
\usepackage{natbib}
\usepackage{url}
\usepackage{a4wide}
\usepackage{amsmath}
\usepackage{amsfonts}
\usepackage{amssymb}
\usepackage{amsthm}
\usepackage{bussproofs}
\usepackage{pxfonts}
\usepackage{cmll}
\usepackage[applemac]{inputenc}
\usepackage[stable]{footmisc}
\usepackage{color}
\newtheorem{de}{Definition}
\newtheorem{theo}[de]{Theorem}
\newtheorem{lem}[de]{Lemma}

\usepackage{eurosym}
\usepackage{todonotes}
\allowdisplaybreaks

\def\bexists{\exists\kern-.39em{}\mbox{I}}   
\def\bforall{\forall\kern-.41em{}\raise.31em\hbox{\vbox{\hrule width .25em}}} 
\def\hyph{\mbox{-}}

\newbox\gnBoxA
\newdimen\gnCornerHgt
\setbox\gnBoxA=\hbox{$\ulcorner$}
\global\gnCornerHgt=\ht\gnBoxA
\newdimen\gnArgHgt
\def\gn #1{%
\setbox\gnBoxA=\hbox{$#1$}%
\gnArgHgt=\ht\gnBoxA%
\ifnum     \gnArgHgt<\gnCornerHgt \gnArgHgt=0pt%
\else \advance \gnArgHgt by -\gnCornerHgt%
\fi \raise\gnArgHgt\hbox{$\ulcorner$} \box\gnBoxA %
\raise\gnArgHgt\hbox{$\urcorner$}}

\begin{document}
\title{A Note on McGee's $\omega$-Inconsistency Result\thanks{After I had finished a draft of this note I discovered that basically the same point has been made in \cite[Theorem 11]{lei01}. This put an end to the idea of publishing the note. Since in the present note the point is made in a slightly different and, I believe, crisper way I decided nonetheless to make the note publicly available. I wish to thank Matteo Zichetti, who drew my attention to a rather important typo in an earlier version of this note.}} 
\author{Johannes Stern\\johannes.stern@lrz.uni-muenchen.de}

\maketitle
\begin{abstract}In this note we show that McGee's $\omega$-inconsistency result can be derived from Löb's theorem.\end{abstract}
In his paper ``How Truthlike Can a Predicate be?" \cite{mcg85} showed the $\omega$-inconsistency of a broad family of theories of truth. The purpose of this note is to highlight the connection between McGee's result and Löb's theorem. Once this connection is made explicit McGee's $\omega$-inconsistency result may be viewed as a variant of Gödel's second incompleteness theorem. For expository purposes we start by providing McGee's result roughly following his original derivation.\footnote{On notation: $\mathcal{L}$ is a standard arithmetical language with the exception that we assume the existence of certain function symbols in the language. In particular we assume the existence of the function symbol $f^\bullet$ (cf.~below). $\mathcal{L}_T$ ($\mathcal{L}_P$) is the extension of $\mathcal{L}$ by a unary predicate $T$ ($P$). We assume some standard coding scheme for the expressions of the languages under consideration and denote the name of the code of an expression $\eta$ by $\gn\eta$. The numeral of a natural number $n$ is denoted by $\overline n$. Finally, we write $\ulcorner\phi(\dot x)\urcorner$ for denoting the function that with $n$ as argument provides the code of the formula $\phi(\overline n)$.} 

\begin{theo}[McGee]\label{OI}Let $\Gamma$ be a theory extending $Q$ in the language $\mathcal{L}_T$, which is closed under the rule
\begin{flalign*}&(T\hyph{Intro})&&\frac{\phi}{T\gn\phi}&&\end{flalign*}
and proves
\begin{flalign*}
&(\text{Cons})&&T\gn{\neg\phi}\rightarrow\neg T\gn\phi&&\\
&(T\hyph\text{Imp})&&T\gn{\phi\rightarrow\psi}\rightarrow(T\gn\phi\rightarrow T\gn\psi)&&\\
&(\text{UInf})&&\forall x T\gn{\phi(\dot{x})}\rightarrow T\gn{\forall v\phi(v)}&&
\end{flalign*}
for all $\phi,\psi\in\mathsf{Sent_{\mathcal{L}_{T}}}$. Then $\Gamma$ is $\omega$-inconsistent.
\end{theo}
A theory is $\omega$-inconsistent if there exists a formula $\phi(x)$ such that the theory proves $\neg\forall x\phi$ and $\phi(\overline n)$ for all $n\in\omega$. In other words, if we allow for one application of the $\omega$-rule inconsistency will arise. The $\omega$-rule allows us to infer $\forall x\phi$ if we have derived $\phi(\overline n)$ for all $n\in\omega$. 

The crucial observation by McGee was that there is a two-place primitive recursive function $f$, which when applied to a natural number $n$ and the code, i.e., the Gödel number, of a sentence $\phi$ provides the code of the sentence
\begin{align*}\underbrace{T\ulcorner\dots T}_{n\hyph\text{times}}\gn\phi\ldots\urcorner.\end{align*}
This allowed McGee to define an $\omega$-truth predicate
\begin{align*}T^\omega x:=\forall y T f^\bullet(y,x)\end{align*}
where $f^{\bullet}$ is a function symbol representing $f$ in $\Gamma$.\footnote{For ease of exposition we avail ourselves to certain function symbols, such as $f^{\bullet}$, in the language. In this we follow \cite{hal11} presentation of McGee's result. \cite{mcg85} presented his result without assuming such function symbols.}
A sentence $\phi$ is $\omega$-true iff each sentence resulting from $\phi$ by applying any finite number of truth predicate to the sentence is true. By first-order logic and (UInf) one can easily prove the following characteristics of $T^\omega$:
\begin{flalign*}
&\text{(A1)}&&T^\omega\gn\phi\rightarrow T\gn{T^\omega\gn\phi}&&\\
&\text{(A2)}&&T^\omega\gn\phi\rightarrow T\gn\phi&&
\end{flalign*}
With these prerequisites we can provide a crisp version of McGee's original proof:
\begin{proof}We start by an application of the Diagonal lemma:
\begin{enumerate}
\item $\gamma\leftrightarrow\neg T^\omega\gn\gamma$
\item $T\gn\gamma\leftrightarrow T\gn{\neg T^\omega\gn\gamma}$\hfill1,($T$-Intro),($T$-Imp)
\item $T\gn\gamma\rightarrow \neg T\gn{T^\omega\gn\gamma}$\hfill 2,(Cons)
\item $T\gn\gamma\rightarrow \neg T^\omega\gn\gamma$\hfill 3,A1
\item $T^\omega\gn\gamma\rightarrow T\gn\gamma$\hfill A2
\item $\neg T^\omega\gn\gamma$\hfill 4,5
\item $\gamma$\hfill 1,6
\end{enumerate}
From line 7 and ($T$-Intro) we may derive 
\begin{align*}\underbrace{T\gn\gamma}_{T f^\bullet(\overline 0,\gamma),}, \underbrace{T\gn{T\gn\gamma}}_{T f^\bullet(\overline 1,\gn\gamma),}, \underbrace{T\gn{T\gn{T\gn\gamma}}}_{T f^\bullet(\overline 2,\gn\gamma),},\ldots\end{align*}
By the $\omega$-rule this yields $\forall x T f^\bullet(x,\gn\gamma)$, that is, $T^\omega\gn\gamma$, which contradicts line 6 above.
\end{proof}

McGee's result adds to a family of inconsistency results, such as, Tarski's undefinability result (\cite{tar35}) or Montague's theorem (\cite{mon63}) that point to severe limitations on \emph{how truthlike predicates can be}. However, as we have seen, we have to go beyond the resources of classical first-order logic to turn the $\omega$-inconsistency result into an inconsistency result proper. But if these resources, that is, the $\omega$-rule, are made available, McGee's result turns out to be a direct consequence of Löb's theorem and as such a variant of Gödel's second incompleteness theorem.\footnote{See \cite{smo84} for a discussion of the relation between Löb's theorem and Gödel's second incompleteness theorem.} The reason is that in a theory of truth that proves ($T$-Imp), (UInf) and is closed under the rule ($T$-Intro), we can derive the three Löb derivability conditions for the $\omega$-truth predicate $T^\omega$ if we allow for application of the $\omega$-rule.\footnote{There exist $\omega$-consistent theories of this kind thus this observation is non-trivial in the sense that the application of the $\omega$-rule does not lead to inconsistency, i.e., the explosion of the derivability relation.} This implies that we can derive Löb's theorem for $T^\omega$, which directly contradicts the principle (Cons) because (Cons) forces the $\omega$-truth predicate to be provably consistent, that is, we can prove $\neg T^\omega\gn{\overline 0=\overline 1}$. 

Let us make this observation explicit. \cite{lob55} showed that if a theory $\Lambda$ extending $Q$ the following conditions are satisfied for a predicate $P$
\begin{flalign*}
&(D1)&&\Lambda\vdash\phi\Rightarrow\Lambda\vdash P\gn\phi&&\\
&(D2)&&\Lambda\vdash P\gn{\phi\rightarrow\psi}\rightarrow(P\gn\phi\rightarrow P\gn\psi)\\
&(D3)&&\Lambda\vdash P\gn\phi\rightarrow P\gn{P\gn\phi}
\end{flalign*}
for all  $\phi,\psi\in\mathsf{Sent_{\mathcal{L}_{P}}}$. Then $\Lambda$ proves
\begin{flalign*}
&(L1)&&P\gn\phi\rightarrow\phi\Rightarrow\Lambda\vdash\phi&&\\
&(L2)&&P\gn{P\gn\phi\rightarrow\phi}\rightarrow P\gn\phi
\end{flalign*}
for all  $\phi\in\mathsf{Sent_{\mathcal{L}_{P}}}$. (L1) is known as Löb's theorem, whereas (L2) is the so-called formalized Löb's theorem. In a theory $\Sigma$ that is just like the theory $\Gamma$ of Theorem \ref{OI} with the exception that (Cons) is no longer assumed we may establish both versions of Löb's theorem for the predicate $T^\omega$. 

\begin{theo}[$\omega$-Löb]\label{LP}Let $\Sigma$ be a theory extending $Q$ in the language $\mathcal{L}_T$, which is closed under the rule
\begin{flalign*}&(T\hyph{Intro})&&\frac{\phi}{T\gn\phi}&&\end{flalign*}
and proves
\begin{flalign*}
&(T\hyph\text{Imp})&&T\gn{\phi\rightarrow\psi}\rightarrow(T\gn\phi\rightarrow T\gn\psi)&&\\
&(\text{UInf})&&\forall x T\gn{\phi(\dot{x})}\rightarrow T\gn{\forall v\phi(v)}&&
\end{flalign*}
for all $\phi,\psi\in\mathsf{Sent_{\mathcal{L}_{T}}}$. Let $T^\omega$ be defined as above. Then
\begin{flalign*}
&(i)&&\Sigma\vdash_\omega T^\omega\gn\phi\rightarrow\phi\Rightarrow\Sigma\vdash_\omega\phi&&\\
&(ii)&&\Sigma\vdash_\omega T^\omega\gn{T^\omega\gn\phi\rightarrow\phi}\rightarrow T^\omega\gn\phi.
\end{flalign*}
\end{theo}
The derivability relation $\vdash_\omega$ signifies the closure of the classical derivability relalion under non-embedded applications of the $\omega$-rule.\footnote{As a consequence we do not appeal to full $\omega$-logic for establishing Theorem \ref{LP}. In other words to avoid triviality we only need to assume the $\omega$-consistency of $\Sigma$ rather than the existence of standard models of $\Sigma$.} As to be expected,Theorem \ref{LP} is established by showing that the three Löb derivability conditions can be proved in $\Sigma$ for the predicate $T^\omega$. We state this claim in the following lemma.

\begin{lem}\label{LC}Let $\Sigma$ be as in Theorem \ref{LP}. Then for all $\phi\in\mathsf{Sent_{\mathcal{L}_T}}$
\begin{flalign*}
&(M1)&&\Sigma\vdash_\omega\phi\Rightarrow\Sigma\vdash_\omega T^\omega\gn\phi&&\\
&(M2)&&\Sigma\vdash_\omega T^\omega\gn{\phi\rightarrow\psi}\rightarrow (T^\omega\gn\phi\rightarrow T^\omega\gn\psi)\\
&(M3)&&\Sigma\vdash_\omega T^\omega\gn\phi\rightarrow T^\omega\gn{T^\omega\gn\phi}.
\end{flalign*}\end{lem}
\begin{proof}(M1) follows directly from the rule ($T$-Intro), the $\omega$-rule and the definition of $T^\omega$. (M2) follows from ($T$-Imp), ($T$-Intro), the $\omega$-rule and again the definition of $T^\omega$. For (M3) we observe that by (A1) we have $T^\omega\gn\phi\rightarrow T\gn{T^\omega\gn\phi}$, i.e., $T^\omega\gn\phi\rightarrow T f^\bullet(\overline 0,\gn{T^\omega\gn\phi})$. By ($T$-Intro) and ($T$-Imp) we obtain $T\gn{T^\omega\gn\phi}\rightarrow T\gn{T\gn{T^\omega\gn\phi}}$. Then by (A1) we derive $T^\omega\gn\phi\rightarrow T\gn{T\gn{T^\omega\gn\phi}}$, that is, $T^\omega\gn\phi\rightarrow T f^\bullet(\overline 1,\gn{T^\omega\gn\phi})$. Clearly, we may repeat this process and therefore, by an application of the $\omega$-rule, derive $T^\omega\gn\phi\rightarrow\forall x T f^\bullet(x,\gn{T^\omega\gn\phi})$. By definition of $T^\omega$ this is the desired (M3).
\end{proof}
By Lemma \ref{LC}, Theorem \ref{LP} is a direct corollary of Löb's original result. Moreover, as we have already mentioned, McGee's $\omega$-inconsistency result proves to be a direct corollary of this result.
\begin{proof}[Alternative proof of Theorem \ref{OI}] Since the theory $\Gamma$ proves (Cons) and thus $\Gamma\vdash_\omega T^\omega\gn{\neg\phi}\rightarrow\neg T^\omega\gn\phi$, we have $\Gamma\vdash T^\omega\gn{\overline 0=\overline 1}\rightarrow 0=1$. By Theorem \ref{LP} Löb's theorem holds for $\Gamma$. This yields the contradiction.
\end{proof}

This observation establishes a firm link between McGee's theorem and Löb's theorem. This connection is of course not too surprising. In their paper ``Possible World Semantics for Modal Notions Conceived as Predicates" \cite{hlw03} already remark: ``In the end all limitative results can be derived from Löb's theorem." Their observation, however, depends on rather involved semantic considerations while the present note might serve as an accessible illustration of this general fact.\footnote{In the terminology of \cite{hlw03}, the connection between McGee's result and Löb's theorem is roughly as follows: a given possible world frame $\langle W,R\rangle$ admits a valuation only if the frame $\langle W,R^\ast\rangle$ does, where $R^\ast$ is the transitive closure of $R$. This is precisely due to the fact that we can define an $\omega$-truth predicate. Indeed $R^*$ is the relevant accessibility relation for $T^w$. Moreover, the valuation on the frame $\langle W,R^\ast\rangle$ must be such that Löb's theorem is true at each world, which immediately yields all limitative results. The question which possible world frames allow for an valuation may therefore be rephrased as the question: for which transitive frames can we find a valuation such that Löb's theorem is true at each world of the model?}

\bibliographystyle{apalike}
\bibliography{/users/jstern/Dropbox/Uni/Bibliography/LITA-3}
\end{document}